\documentclass[12pt]{amsart}
\usepackage{amssymb,amsmath,color}
\usepackage{dsfont}
\usepackage{enumitem,amsrefs,hyperref}
\usepackage{color}
\usepackage{tikz}
\usepackage{pgf}
\usepackage{pgfplots}
\usepackage{bm}
\usetikzlibrary{calc, positioning}
\usepackage{fullpage}
\usepackage{accents}

\newtheorem{thm}{Theorem}
\newtheorem{problem}[thm]{Problem}
\newtheorem{prop}[thm]{Proposition}

\newtheorem{lem}[thm]{Lemma}
\newtheorem{cor}[thm]{Corollary}
\newtheorem{conj}[thm]{Conjecture}
\theoremstyle{definition}
\newtheorem{ex}[thm]{Example}
\newtheorem{exd}[thm]{Example and Definition}
\newtheorem{exer}[thm]{Exercise}
\newtheorem{defi}[thm]{Definition}
\newtheorem{quest}[thm]{Question}
\theoremstyle{remark}
\newtheorem{rmk}[thm]{Remark}
\numberwithin{thm}{section}

\begin{document}
\title{Lectures on Nonnegative Polynomials and Sums of Squares}
\author{Grigoriy Blekherman}
\author{Jannik Wesner}
\maketitle

\section*{Introduction}

These lecture notes provide an informal introduction to the theory of nonnegative polynomials and sums of squares. We highlight the history and some recent developments, especially the new connections with classical (complex) algebraic geometry and commutative algebra. We briefly discuss the connection to semidefinite programming and applications. We expect the reader to be familiar with algebra and fundamentals of algebraic geometry, on the level of a graduate course, and know basic facts about positive semidefinite matrices. For additional background we refer to \cite{MR1322960} for commutative algebra, \cite{MR1940576} for convex geometry and \cite{BPT} for polynomial optimization.

\section{Lecture 1}\label{sec1}

We begin with a motivating question. Suppose that we are given a partially specified matrix $A$ as in Example \ref{ex:psmat}. Can we complete $A$ by finding the unspecified ``?" entries, so that the full matrix is positive semidefinite?


\begin{ex}\label{ex:psmat}
	\begin{enumerate}
	\item The matrix
		\[ \begin{pmatrix}
			1 & 1 & 1 & ?\\
			1 & 1 & 1 & ?\\
			1 & 1 & 1 & 1\\
			? & ? &  1 & 1
		\end{pmatrix}\]
		can be completed to a full positive semidefinite matrix by setting all ``?" entries to $1$.
	\item The matrix
		\[ \begin{pmatrix}
			1 & 1 & 1 & ?\\
			1 & 1 & 1 & ?\\
			\cline{3-4}
			1 & 1 & \multicolumn{1}{|c}{1} & 2\\
			? & ? &  \multicolumn{1}{|c}{2} & 1
		\end{pmatrix}\]
		cannot be completed to a full positive semidefinite matrix since the $2\times2$ minor in the lower right corner is negative.
	\end{enumerate}
\end{ex}

The second example suggests a necessary condition:
\begin{center}($*$) All fully specified principal minors of $A$ must be nonnegative.\end{center}

Now we consider a more general situtation: 
\begin{exd} \label{ex_graph}
	For a partially specified matrix 
	\[ A = \begin{pmatrix}
		* & * & * & ?\\
		* & * & * & ?\\
		* & * & * & *\\
		? & ? & * & *
	\end{pmatrix}\]
	with known entries ``$*$" and unspecified entries ``$?$" we construct a graph $G$.
	The vertices will be $1,2,3$ and $4$, corresponding to the rows (or columns) of $A$. Vertex $i$ is connected to $j$ if and only if $a_{ij} = *$ is a specified entry. 
	For the matrix $A$ above we arrive at the following graph:\\
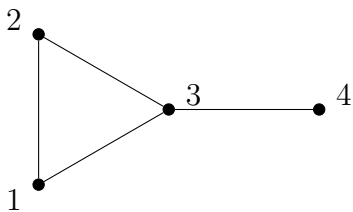
	\begin{center}
	\begin{tikzpicture}
		\tikzset{enclosed/.style={draw, circle, inner sep=0pt, minimum size=.15cm, fill=black}}
		
		\node[enclosed, label={left, yshift=.2cm: 2}] (2) at (0.75,3.00) {};
		\node[enclosed, label={left, yshift=-.2cm: 1}] (1) at (0.75,1.00) {};
		\node[enclosed, label={right, yshift=.2cm: 3}] (3) at (2.48,2.00) {};
		\node[enclosed, label={right, yshift=.2cm: 4}] (4) at (4.48,2.00) {};
		
		\draw (1) -- (2);
		\draw (2) -- (3);
		\draw (3) -- (1);
		\draw (3) -- (4);
	\end{tikzpicture}
	\end{center}

We only consider partially specified matrices where all diagonal entries are specified. Such patterns of specified and unspecified entries in a symmetric matrix are in bijective correspondence to simple graphs.
\end{exd}

We arrive at a general question about patterns of specified and unspecified entries, or equivalently the corresponding graphs:

\begin{quest}\label{quest:gr} For which graphs $G$ is the obvious necessary condition ($*$) also sufficient for positive semidefinite completion? Equivalently, for what patterns of specified and unspecified entries can any assignment of *-entries in $A$ such that all fully specified principal minors are nonnegative, be completed to a full positive semidefinite matrix? 
\end{quest}

\begin{rmk}
Problems where we are given a partially specified matrix and we want to find the unspecified entries so that the full matrix has certain properties are called \textit{matrix completion problems}. The problem we are concerned with is the \textit{positive semidefinite matrix completion problem}. We may also be interested in finding a positive semidefinite completion of minimal possible rank. Such problems are called \textit{low-rank matrix completion problems}. They are related to the area of \textit{sparse recovery and compressed sensing} \cite{laurent2009matrix, MR2680543}. 
\end{rmk}

\subsection{Through the looking glass} We can reformulate our Question \ref{quest:gr} in a different way. Let $I=I(G)$ be the monomial ideal generated by the products $x_i x_j$ where $(i,j)$ is a non-edge
in the graph $G$.

\begin{exer}
\begin{enumerate}

\item Show that the ideal $I(G)$ is radical.

 \item Show that the variety $X=X(G)$ defined by vanishing of $I$ is the union of linear coordinate subspaces, each corresponding
to a maximal clique of $G$, or a maximal fully specified submatrix of $A$ respectively (For a generalization of this phenomenon see Stanley-Reisner ideals \cite{MR1453579}).
\end{enumerate}
\end{exer}
 
In Example \ref{ex_graph} $I=\left< x_1 x_4, x_2 x_4 \right>$. The variety defined by $I$ is $X = L_1 \cup L_2$ where
$L_1 = \mathrm{span}\left\{e_1,e_2,e_3\right\}$ and $L_2 = \mathrm{span}\left\{e_3,e_4\right\}$.
Partially specified matrices are simply quadratic forms modulo $I$. Since $I$ is a radical ideal, we can also think of them as functions on the variety $X$.

The obvious necessary condition ($*$) now says that we are considering quadratic forms that are nonnegative on $X$.

\begin{exer}
Show that a quadratic form nonnegative on $\mathbb{R}^n$ is a sum of squares of linear forms. (Hint: The matrix of a nonnegative quadratic form is positive semidefinite. Now use diagonalization of quadratic forms).
\end{exer}

We can reformulate Question \ref{quest:gr} as a problem about the relationship between nonnegative quadratic forms and sums of squares on $X$.

\begin{quest}\label{quest:ref}
Let $I$ be a square-free quadratic monomial ideal and $X$ be the variety defined by $I$. For what ideals $I$ can any quadratic form nonnegative on $X$ be written as a sum of squares modulo $I$ (we will call such forms sums of squares on $X$)? 
\end{quest}

To study this question and its generalizations we will introduce a little bit more notation. A variety $X\subseteq \mathbb{C}\mathbb{P}^n$ is called \textit{real} if it is defined by real equations, or equivalently if it is invariant under conjugation. However, we would like to exclude varieties such as the one given by $x_0^2+\dots+x_n^2=0$, which contains no real points. A complex variety is \textit{totally real} if it is conjugation invariant and its real points $X(\mathbb{R})=X\cap \mathbb{R}\mathbb{P}^n$ are Zariski dense in $X$. Equivalently $X$ is totally real if it contains a smooth real point, or if the ideal of $X$ is real-radical \cite[Chapters 3 and 4]{MR1659509}.
  
For a real projective variety $X$ we will use $R(X)$ (or sometimes just $R$ when the variety is clear form context) to denote its real \textit{homogeneous coordinate ring} $\mathbb{R}[x_0,\dots,x_n]/I(X)$. Degree $d$ part of the coordinate ring will be denoted by $R(X)_d$ or $R_d$.

\begin{defi}
	Let $I$ be a square-free quadratic monomial ideal in $\mathbb{R}\left[x_1, \dots , x_n\right]$ and let $G$ be the graph on $n$ vertices
	where monomials of $I$ are the non-edges of $G$. Let $X$ be the real projective variety defined by $I$.
	We define 
	\begin{align*}
		P_X &:= P(G) :=
			\left\{ f \in R(X)_2 \; \middle| \; \forall [v] \in X(\mathbb{R}): \, f\left(v\right) \geq 0\right\}, \\
		\Sigma_X &:= \Sigma(G) :=
			\left\{ f \in R(X)_2 \; \middle| \; \exists k \in \mathbb{N} \, \exists \ell_1, \dots \ell_k \in R(X)_1 :
				f = \sum_{i=1}^k \ell_i^2  \right\}.
	\end{align*}
\end{defi}
It is clear that we have the inclusion $\Sigma(G) \subseteq P(G)$. The opposite inclusion is what Question \ref{quest:ref} asks about. We can restate it with our new notation:

\begin{quest}\label{quest:ref2} For which graphs $G$ does $P(G) = \Sigma (G)$ hold?
\end{quest}

Before giving the answer to the above questions, we will give some history of the topic of nonnegative polynomials and sums of squares, which is a classical question in real algebraic geometry.

\subsection{Eine kleine Geschichte}

We will first discuss the question of global nonnegativity. We will call a polynomial $p$ nonnegative if it takes only nonnegative values, i.e. $p(x)\geq 0$ for all $x \in \mathbb{R}^n$. We will call $p$ a sum of squares if it can be written as a sum of squares of polynomials. It is clear that all sums of squares are nonnegative. To study the relationship between nonnegative polynomials and sums of squares it suffices to consider homogeneous polynomials.

\begin{exer}
Define homogenization $\tilde{p}$ of a polynomial $p(x_1,\dots,x_n)$ of degree $d$ by introducing a new variable $x_0$ and multiplying all monomials in $p$ by a power of $x_0$, so that all monomials have degree $d$. More formally:
$$\tilde{p}=x_0^d\cdot p\left(\frac{x_1}{x_0},\dots,\frac{x_n}{x_0}\right).$$
Show that $p$ is nonnegative if and only if $\tilde{p}$ is, and the same is also true for sums of squares.
\end{exer}

\begin{defi}
	Let ${P}_{n,2d}$ denote the set of nonnegative homogeneous polynomials in $n$ variables of degree $2d$.
	Likewise $\Sigma_{n,2d}$ denotes the set of homogeneous polynomials in $n$ variables of degree $2d$, which are sums of squares.
\end{defi}

The question of the relationship between nonnegative polynomials and sums of squares arose during the Ph.D. defense of Hermann Minkowski. David Hilbert was one of the examiners and Minkowski claimed that there exist nonnegative polynomials that are not sums of squares, although he did not prove this claim. Three years later Hilbert published a paper \cite{MR1510517} completely describing all the cases, in terms of degree and number of variables, where nonnegative polynomials are equal to sums of squares.


\begin{thm}\cite{MR1510517}
	The equality ${P}_{n,2d} = \Sigma_{n,2d}$ holds only in the following three cases:
	\begin{enumerate}[label=(\roman*)]
		\item $n=2$ (univariate non-homogeneous polynomials)
		\item $2d = 2$ (quadratic forms)
		\item $n=3$, $2d = 4$ (ternary quartics)
	\end{enumerate}
\end{thm}
\begin{rmk}
	The two smallest cases of interest for strict inclusion of $\Sigma_{n,2d}$ into $P_{n,2d}$ are:
	\begin{enumerate}[label=(\roman*)]
		\item $n=3$, $2d = 6$ and
		\item $n=4$, $2d=4$.
	\end{enumerate}
	One example of a nonnegative polynomial that is not a sum of squares for each of these cases is enough (see Exercise \ref{exer:neq} below). Hilbert's proof however was not constructive and for a long time no explicit examples were known.
\end{rmk}
\begin{exer}\label{exer:neq}
Show the following:
\begin{enumerate}
\item If $\Sigma_{k,2d}\subsetneq P_{k,2d}$ then $\Sigma_{n,2d}\subsetneq P_{n,2d}$ for all $n\geq k$.

\item If $\Sigma_{n,2m}\subsetneq P_{n,2m}$ then $\Sigma_{n,2d}\subsetneq P_{n,2d}$ for all $d\geq m$.
\end{enumerate}  
Conclude that $\Sigma_{3,6} \subsetneq P_{3,6}$ and $\Sigma_{4,4}\subsetneq P_{4,4}$ suffice to settle all strict inclusion cases of Hilbert's Theorem. These arguments were part of Hilbert's original proof. (Hint: For the second part take $f\in P_{n,2m}$, $f\notin \Sigma_{n,2m}$ and consider $f\cdot x_1^2$. Argue that $f\in  P_{n,2m+2}$ and $f \notin \Sigma_{n,2m+2}$.)
\end{exer}

The first explicit example of a nonnegative polynomial that is not a sum of squares was found by Motzkin in 1967 \cite{MR0223521}. In fact Motzkin was not aware that this was an open problem. Olga Taussky-Todd, who was present during the seminar in which Motzkin described his construction, later notified him that he found the first example of a nonnegative polynomial that is not a sum of squares \cite{MR1747589}. 
\begin{ex}[Motzkin's Form]
	$M(x,y,z)=x^2y^4 + x^4 y^2 + z^6 - 3 x^2 y^2 z^2 \in P_{3,6}$ by the arithmetic mean/geometric mean inequality \cite{MR944909}. In Exercise \ref{exer:np} below you will learn a way to show that $M(x,y,z)$ is not a sum of squares.

\end{ex}

In the following Exercises we will develop a general method to show that a form is not a sum of squares, based on the monomials that occur in the form. This method can also be used to find an explicit example for the case $(4,4)$.  These ideas are originally due to Choi, Lam, and Reznick \cite{MR1327293}. 

\begin{exer}\label{exer:np}
The \textit{Newton Polytope}\index{Newton polytope} $N_p$ of a polynomial $p$ is the convex hull of the vectors of monomial exponents that occur in $p$. For example, the Newton Polytope of $x^2y+xy+1$ is the convex hull of vectors $(2,1), (1,1)$ and $(0,0)$. Show that if a polynomial $p=\sum_i q^2_i$ is a sum of squares, then the Newton Polytope of each $q_i$ is contained in $\frac{1}{2}N_p$. (Hint: Consider the convex hull of the Newton polytopes of $q_i$'s.)
\end{exer}

\begin{exer}Use Exercise \ref{exer:np} to show that the Motzkin form is not a sum of squares.

\end{exer}

\begin{exer}\cite{MR498384}
   Show that the form $S(x,y,z)=x^4y^2+y^4z^2+z^4x^2-3x^2y^2z^2$ is nonnegative but not a sum of squares.
\end{exer}

\subsubsection{Hilbert's 17th Problem:} After showing that nonnegative polynomials are not necessarily sums of squares of polynomials, Hilbert showed in 1893  that in $3$ variables every nonnegative form is a sum of squares of rational functions \cite{MR1554835}. Hilbert's 17th problem asked whether the same holds in any number of variables. There is a more convenient way of describing what it means for a form $p$ to be a sum of squares of rational functions. Let $p = \sum \left(\frac{q_i}{r_i}\right)^2$ for some $q_i,r_i \in \mathbb{R}\left[x\right]$. Clearing the denominators one arrives at
$p r^2 = \sum {q'}_i^2$. In other words a form is a sum of squares of rational functions if and only if there exists a form $r$ such that $p\cdot r^2$ is a sum of squares of forms. We can even allow a sum of squares multiplier, instead of a single square. If we assume that $p \cdot \left(\sum h_i^2\right) = \sum q_i^2$ then multiplying both sides by $\sum h_i^2$ we  also arrive at the form
\[ p \cdot \left(\sum h_i^2\right)^2 = \underbrace{\left(\sum q_i^2\right)\cdot\left(\sum h_i^2\right)}_{\text{sum of squares}} .\]
Therefore we can ask instead: Is there a multiplier $\sum h_i^2$ such that $p \cdot \left(\sum h_i^2\right)$ is a sum of squares.\\

Hilbert's 17th problem was solved in 1920's by Artin using the Artin-Schreier theory of real closed fields \cite{MR1659509}. He showed that any nonnegative form is a sum of squares of rational functions. However, some quantitative aspects of Hilbert's 17th problem are not well-understood. If we know the degree of the sums of squares multiplier $h$, then the sum of squares decomposition can still be computed via semidefinite programming \cite[Chapter 3]{BPT}. Even now, both lower and upper degree bounds for sums of squares multipliers are not well-understood. The best current upper bounds were found in \cite{lombardi2014elementary}. See \cite{MR3545483}  for some lower bounds. For the case $n=3$ the best bounds come from Hilbert's original proof. See \cite{MR3998790} for a modern proof of this result, as well as, tight bounds for the degree of sums of squares multipliers on curves. 

\begin{problem}
Develop a better understanding of degree bounds on sums of squares of rational functional in Hilbert's 17th problem and, more generally, for nonnegative polynomials on a totally real variety $X$.
\end{problem}

\begin{rmk}
We can also consider fixing a multiplier form, for instance, $(x_1^2+\dots+x_n^2)^k$ and given a nonnegative form $p$ ask for the smallest value of $k$ such that $p\cdot (x_1^2+\dots+x_n^2)^k$ becomes a sum of squares. Such a multiplier will work for any positive form, but unfortunately there exist forms for which no positive sum of squares multiplier will work. For more on this line of research see \cite{MR1747589, MR2159759, MR1347159, MR2898748}.
\end{rmk}


\subsection{Two is enough.} To understand all cases of equality between degree $2d$ nonnegative forms on a real projective variety $X$ it suffices to understand all cases of equality between quadratic forms on varieties, i.e. it suffices to understand all varieties $X$ for which $P_X=\Sigma_X$, where $P_X$ is the set of all quadratic forms nonnegative on $X$, and $\Sigma_X$ is the set of sums of squares of linear forms. To understand why this is true we introduce the Veronese map.

The $d$-th Veronese map $\nu_d$ is a map from $\mathbb{P}^n$ to $\mathbb{P}^{\binom{n+d}{d} - 1}$ which sends a point with projective coordinates $\left[x_0:\dots : x_n\right] $ to all monomials of degree $d$ in variables $x_0,\dots,x_n$:
\begin{align*}
	\nu_d \, : \, \mathbb{P}^n &\hookrightarrow \mathbb{P}^{\binom{n+d}{d} - 1} \\
	\left[x_0:\dots : x_n\right] &\mapsto \left[\text{all monomials of degree $d$ in these variables}\right]
\end{align*}
Here is an explicit example of the second Veronese map from $\mathbb{P}^2$ to $\mathbb{P}^5$:
$$\nu_2 \, : \, \mathbb{P}^2 \hookrightarrow \mathbb{P}^5, \; \left[x_0:x_1:x_2\right] \mapsto \left[x_0^2:x_0x_1:x_0x_2:x_1^2:x_1x_2:x_2^2\right].$$
The map $\nu_d$ is very useful because it ``linearizes" degree $d$ forms on $X$: degree $d$ forms on $X$ are simply linear forms on $\nu_d(X)$. 
Therefore the following identification can be made:
Degree $d$ forms on $X$ are linear forms on $\nu_d\left(X\right)$; degree $2d$ forms on $X$ are quadratic forms on $\nu_d\left(X\right)$. If we use $P_{X,2d}$ to denote the set of all degree $2d$ forms nonnegative on $X$, and $\Sigma_{X,2d}$ for the set of sums of squares of degree $d$ forms, then we see that: $$P_{X,2d} = \Sigma_{X,2d} \hspace{1cm} \text{if and only if}   \hspace{1cm} P_{\nu_d\left(X\right), 2} = \Sigma_{\nu_d\left(X\right), 2}.$$

\subsection{Monomials strike back.}

We now come back to the case of quadratic monomial ideals and the associated graphs. We first define an important class of graphs:
\begin{defi}
	A graph $G$ is called \textit{chordal}, if any cycle of length at least $4$ in $G$ has a chord dividing it. In other words $G$ has no chordless induced
	cycles of length at least $4$.
\end{defi}

Chordal graphs, and the related concept of \textit{treewidth} have many interesting properties and characterizations, they are an algorithmically important class of graphs \cite{vandenberghe2015chordal}. Here is a different characterization of chordal graphs:

\begin{exer}\cite{MR130190}
A \textit{clique sum} of two graphs $G$ and $H$ is a gluing of $G$ and $H$ along a complete subgraph contained in both $G$ and $H$. Show that a graph $G$ is chordal if and only if $G$ is the clique-sum of complete graphs.
\end{exer}

We are ready to answer Question \ref{quest:gr} from the beginning of the lecture, and the equivalent Questions \ref{quest:ref} and \ref{quest:ref2}.
\begin{thm}\cite{MR739282}\label{thm:graphs}
	Let $G$ be a graph and $X$ the associated variety. Then $\Sigma(G)=P(G)$ if and only if $G$ is chordal.
	In this case any $p \in  P(G)$ can be written as sum of at most $\mathrm{dim}\,X + 1$ squares (mod $I$).
	This number is equal to the size of the maximal specified minor in the respective matrix, which is equal to $\omega\left(G\right)$, the clique number of $G$.
\end{thm}


In combinatorial commutative algebra there is a theorem due to Fr\"oberg, which gives a different characterization of chordal graphs. Before we state the theorem we need to discuss syzygies and free resolutions. 

\subsubsection{Aside on syzygies and free resolutions} Let us begin by considering two quadratic forms $p$ and $q$ in $3$ variables, or equivalently two quadric curves in $\mathbb{P}^2$. We would like to understand all linear relations between $p$ and $q$ with coefficients in $\mathbb{R}[x_1,x_2,x_3]$. More plainly, we would like to understand for what polynomials $r_1$ and $r_2$ in $\mathbb{R}[x_1,x_2,x_3]$ we have $pr_1+qr_2=0$. Such a linear relation is called a \textit{syzygy} of $p$ and $q$. If either $p$ or $q$ is irreducible then it is easy to see that there is only one minimal syzygy that generates all others: $p\cdot q+q\cdot(-p)=0$. Since $p$ and $q$ are quadratics, this is a quadratic syzygy between $p$ and $q$. Special choices of $p$ and $q$ can lead to linear syzygies. For instance if $p=xy$ and $q=yz$ then there is a linear syzygy $p\cdot z+q\cdot (-x)=0$. Notice that this linear syzygy corresponds to the special geometry of the intersection of the curves defined by $p$ and $q$, namely they intersect in a positive dimensional set. For more on the fascinating connection between syzygies and geometry we refer the reader to \cite{MR2103875}.

A \textit{minimal free resolution} of the ideal $I$ (or essentially equivalently of the quotient $S/I$) considers not only syzygies between the generators of the ideal, but also syzygies between syzygies, syzygies between syzygies of syzygies and so on. We illustrate the idea of minimal free resolution by an example: let $G$ be a cycle on $5$ vertices with edges $(1,3)$, $(1,4)$, $(2,4)$, $(2,5)$, and $(3,5)$ and let $I=I(G)=\langle x_1x_2,x_2x_3,x_3x_4,x_4x_5,x_1x_5\rangle $. The first syzygies of $I$ are generated by syzygies of the type $x_1x_2\cdot x_3+x_2x_3\cdot(-x_1)=0$ between the five pairs of adjacent generators. By numbering the generators we can arrange the syzygies as $5$ dimensional vectors. For instance the above syzygy would we written as $(x_3,-x_1,0,0,0)$ and the next syzygy as $(0,x_4,-x_2,0,0)$. Together these five syzygies generate the \textit{first syzygy module}. Now we need to consider syzygies between syzygies. In this case there is a unique such syzygy: $$x_4x_5\begin{pmatrix} x_3\\-x_1\\0\\0\\0\end{pmatrix}+x_1x_5\begin{pmatrix} 0\\x_4\\-x_2\\0\\0\end{pmatrix}+x_1x_2\begin{pmatrix} 0\\0\\x_5\\-x_3\\0\end{pmatrix}+x_2x_3\begin{pmatrix} 0\\0\\0\\x_1\\-x_4\end{pmatrix}+x_3x_4\begin{pmatrix} -x_5\\0\\0\\0\\x_2\end{pmatrix}=0.$$
The resolution now stops since this unique syzygy cannot lead to any further syzygies. Let $S$ denote the polynomial ring $\mathbb{R}[x_1,\dots,x_5]$. The syzygy exploration above can be arranged into an exact sequence as follows:
\[ 0\rightarrow S(-5)\xrightarrow{\begin{pmatrix}x_4x_5\\x_1x_5\\x_1x_2\\x_2x_3\\x_3x_4\end{pmatrix}} S(-3)^5\xrightarrow{\begin{pmatrix} x_3&0&0&0&-x_5\\-x_1&x_4&0&0&0\\0&-x_2&x_5&0&0\\0&0&-x_3&x_1&0\\ 0&0&0&-x_4&x_2\end{pmatrix}} S(-2)^5\xrightarrow{\begin{pmatrix}x_1x_2\\x_2x_3\\x_3x_4\\x_4x_5\\x_1x_5\end{pmatrix}^T} S\rightarrow S/I\rightarrow 0.\]
The numbers in parenthesis such as $S(-5)$ denote a shift in the grading, so that all maps send elements of degree $0$ to elements of degree $0$. They can be safely disregarded for our purposes.

In recent years some connections between the study of nonnegative polynomials and sums of squares on a real variety $X$ and the properties of its free resolution have been discovered. This is a fascinating and developing direction of research. 

For these lectures we will only consider ideals generated by quadrics. The following properties of linear syzygies will be of interest to us:

\begin{defi}
Let $I$ be an ideal of $S=\mathbb{R}[x_1,\dots,x_n]$ generated by quadrics. The\textit{ Green-Lazarsfeld index} of $S/I$ is the largest number of steps plus one in the minimal free resolution of $S/I$ such that all of the syzygies in these steps are linear. It is equal to $1$ if $I$ is generated by quadrics but quadrics have nonlinear minimal syzygies, it is $2$ if the generating quadrics have only linear minimal syzygies, but these syzygies have non-linear syzygies, etc.

The quotient ring $S/I$ has \textit{Castelnuovo-Mumford regularity} $2$ if all of the syzygies in the minimal free resolution of $S/I$ are linear.

The \textit{length of the linear strand} in the minimal free resolution of $S/I$ is the number of steps that linear syzygies of quadrics persist plus one: it is $1$ if the quadrics in $I$ have no linear syzygies, it is $2$ if the quadrics have linear syzygies, but these syzygies have no linear syzygies, etc.
\end{defi}
 
 In the above example the Green-Lazarsfeld index is $1$ since all the entries in the $5\times 5$ matrix are linear, but the following matrix has quadratic entries. Therefore the Castelnuovo-Mumford regularity is greater than $2$. The length of the linear strand is two, since quadrics of $I$ have linear syzygies, but the resulting linear syzygies have no linear syzygies among them. 
\subsection{Back to the future.}
Now we can state Fr\"oberg's theorem.
\begin{thm}\cite{MR1171260}
	A quadratic, squarefree monomial ideal $I$ has Castelnuovo-Mumford-regularity $2$ if and only if the corresponding graph $G$ is chordal.
\end{thm}



We arrive at a common generalization of Hilbert's 1888 theorem and Theorem \ref{thm:graphs} which describes all cases of equality between nonnegative polynomials and sums of squares on a totally real variety.

\begin{thm}\cite{BPSV,MR3633773,MR3486176} \label{thm_reg_2}
	Let $X$ be a totally real variety in $\mathbb{P}^n$. Then the following equivalence holds:
	\[ P_X = \Sigma_X \iff X \text{ has regularity } 2
	 \]
	If $P_X = \Sigma_X$ holds, then we need at most $\dim X+1$ many squares for writing an element in $P_X$ as
	a sum of squares modulo $I(X)$.
\end{thm}

In the next lecture we will show some parts of Theorem \ref{thm_reg_2} for irreducible $X$. Irreducible $2$-regular varieties have been classified by Bertini and Del Pezzo under a different name. We will need some definitions first. 

\begin{defi}
A projective variety $X\subseteq \mathbb{P}^n$ is called \textit{non-degenerate}, if $X$ is not contained in a hyperplane in $\mathbb{P}^n$.
\end{defi}

\begin{exer}\label{ex:degree}
Let $X\subset \mathbb{C}\mathbb{P}^n$ be a non-degenerate irreducible complex projective variety. Show that $$\deg X\geq \operatorname{codim} X+1.$$ (Hint: Consider projecting away from a general point of $X$; see Theorem \ref{thm:proj}).
\end{exer}
 
An irreducible, non-degenerate variety is called a \textit{variety of minimal degree} if $\deg \left(X\right) = \operatorname{codim}\left(X\right) + 1$.

\begin{rmk}
	By a result of Eisenbud and Goto \cite{MR741934}, an irreducible projective variety $X$ has regularity $2$ if and only if $X$ is a variety of minimal degree.
\end{rmk}

We now take a brief detour to discuss the connection between sums of squares and semidefinite programming. This connection is of paramount importance in applications.

\subsection{Sums of Squares and Semidefinite Programming}
It is very important in applications that sums of squares provide a satisfactory certification of nonnegativity not only philosophically, but also algorithmically. This is due to the fact that testing whether a polynomial is a sum of squares can be done via \textit{semidefinite programming} \cite{BPT}. We first discuss semidefinite programming.
\subsubsection{A Gentle Introduction to Semidefinite Programming} Semidefinite programming is a generalization of linear programming where we also allow positive semidefiniteness constraints. More precisely, a general semidefinite program has the form:

\begin{align*}
& \underset{X}{\text{minimize}}
& & \langle C,X\rangle \\
& \text{subject to}
& & \langle A_i,X\rangle=b_i \; \; i = 1, \ldots, m,\\
& & & X \succeq 0.
\end{align*}

The inner product is the trace inner product $\langle A,B\rangle =\operatorname{tr} AB$. Geometrically, semidefinite programming problem optimizes a linear functional over an (affine) linear slice of the cone of positive semidefinite matrices given by the constraints $ \langle A_i,X\rangle=b_i$ . If $X$ is restricted to being a diagonal matrix then the above semidefinite program becomes a linear program, since positive-semidefiniteness constraint only ensures that the diagonal entries are nonnegative. 

Semidefinite programs can be efficiently solved using interior point methods \cite{MR1857264}. However, these methods will struggle when the matrix size grows large, and this leads to an important bottleneck for sums of squares methods. 
 
 We now demonstrate the connection between sums of squares and semidefinite programming via a small example.
  
 \begin{ex}
	Consider the univariate polynomial $\left(a + b x + c x^2\right)^2$.
	If we set $\mathbf{x} = \left(1,x,x^2\right)^T$ and $\mathbf{v} = \left(a,b,c\right)^T$ it can be written as
	\[ \left(a + b x + c x^2\right)^2 = \mathbf{v}^T \mathbf{x} \cdot \mathbf{v}^T \mathbf{x}
		= \mathbf{x}^T \hspace{-10pt} \underbrace{\left(\mathbf{v} \mathbf{v}^T \right)}_{\text{rank $1$ matrix}} \hspace{-8pt} \mathbf{x} \]
	where
	\[ A := \mathbf{v} \mathbf{v}^T = \begin{pmatrix}
		a^2	& ab		& ac\\
		ab	& b^2	& bc\\
		ac	& bc		& c^2\\
	\end{pmatrix} .\]
\end{ex}

This leads to the correspondence
\begin{align*}
	\text{squares } &\leftrightarrow \text{ Rank $1$  positive semidefinite matrices}\\
	\text{sums of squares } &\leftrightarrow \text{positive semidefinite matrices}.
\end{align*}


We now make this correspondence more explicit. Let $p(x)$ be the polynomial $1 + x + 3x^2 - x^3 + x^4$. To write $p$ as a sum of squares we need to find a posiitive semidefinite matrix $A$ such that
	\[ p = \begin{pmatrix}1 & x & x^2 \end{pmatrix}
		\underbrace{\begin{pmatrix}
			a_{11} & a_{12} & a_{13}\\
			a_{12} & a_{22} & a_{23}\\
			a_{13} & a_{23} & a_{33}
		\end{pmatrix}}_{= A \succcurlyeq 0}
		\begin{pmatrix} 1 \\ x \\ x^2 \end{pmatrix} .\]
	The sums over the antidiagonals give us the coefficients of $p$:
	\begin{align*}
		1 &: \; a_{11} = 1\\
		x &: \; 2 a_{12} = 1\\
		x^2 &: \, 2 a_{13} + a_{22} = 3\\
		x^3 &: \, 2 a_{23} = -1\\
		x^4 &: \, a_{33} = 1
	\end{align*}
	
	By considering the determinant of $A$ one can check that $A$ is positive semidefinite for $\frac{1}{4}(5-\sqrt{5})\leq a_{13} \leq -1$ and the choices of $a_{13}=-1$ and $a_{13}=\frac{1}{4}(5-\sqrt{5})$ will lead to rank $2$ matrices, i.e. ways of writing $p(x)$ as a sum of two squares. 

We see from the above discussion that the question "Is a polynomial a sum of squares ?" is equivalent to "Does there exists a positive semidefinite matrix with certain linear constraints on its entries?". This is precisely the question that is addressed by semidefinite programming. Moreover, the least number of squares needed for a representation as a sum of squares is the least rank of an admissible positive semidefinite matrix.



\section{Lecture 2}

For this lecture let $X$ be a totally real, non-degenerate projective variety.
We let $ R := \mathbb{R}\left[x_0, \dots , x_n\right]/I(X)$ be the homogeneous coordinate ring of $X$.
As before $P_X, \Sigma_X \subset R_2$ are the set of quadratic forms nonnegative on $X$, and the set of sums of squares of linear forms on $X$
(mod $I(X)$). 

\begin{exer} Show that
$P_X$ and $\Sigma_X$ are closed, full-dimensional, pointed, (i.e. $p, -p \in \Sigma_X$ implies $p=0$) convex cones in $R_2$.
\end{exer}

We now discuss some parts of the above exercise to get more intuition on convex geometry of cones $\Sigma_X$ and $P_X$.
We can identify $\Sigma_{\mathbb{P}^n}=P_{\mathbb{P}^n}$ with the cone of $(n+1)\times (n+1)$ positive semidefinite matrices  $\mathcal{S}_+^{n+1}$, which is a full-dimensional cone in the vector space of symmetric matrices. By definition of $\Sigma_X$ it is the projection of $\mathcal{S}_+^{n+1}$ with kernel $I(X)$ and therefore it is full-dimensional in $R_2$. Thus, the larger cone $P_X$ is also full-dimensional. It may be tempting to conclude that closedness of $\Sigma_X$ also follows immediately since it is a projection of a closed cone. However, projection of a closed cone may fail to be closed.

\begin{exer}
	The projection of a closed cone is not closed in general. Let $\pi$ denote the projection map and $C$ a closed cone in $\mathbb{R}^n$. Show that the following holds:
	\begin{enumerate}[label=(\alph*)]
		\item If $C \cap \ker \pi = \left\{0\right\}$ , then $\pi\left(C\right)$ is closed.
		\item If $C \cap \ker \pi$ contains an interior point of $C$, then $\pi\left(C\right)$ is a vector space.
		\item If $C \cap \ker \pi \subset \partial C$, then construct an example where $\pi\left(C\right)$ is closed, and an example where it is not closed.
	\end{enumerate}
\end{exer}

We claim that since $X$ is non-degenerate and totally real the first case holds: $\mathcal{S}_+^{n+1} \cap I(X) = \left\{0\right\}$. Suppose this is not true. Then there exists $q \in \Sigma_{\mathbb{P}^n} \cap I(X)$.
By definition of $\mathcal{S}^{n+1}_+$ there are linear forms $\ell_i$ such that $q = \sum \ell_i^2$.
It follows $X \subset \mathcal{V}\left(\sum_i \ell_i^2 \right)$. Since $X$ is totally real and
\[ \sum_i \ell_i^2 \left(p\right) = 0 \iff \left( \forall i : \; \ell_i\left(p\right) = 0 \right) \]
for real points $p$, we arrive at the contradiction $X \subset \cap_i \mathcal{V}\left(\ell_i\right)$.

We will gain some insight into the geometry of $P_X$ and $\Sigma_X$ by \textit{projecting} $X$ \textit{away from a point} $v\in X$. Such a projection is often called an \textit{inner projection of $X$}. If $v\notin X$ the projection away from $v$ is called an \textit{outer projection} of $X$. We summarize some properties of inner projections below:

\begin{thm}\label{thm:proj}\cite{MR1416564}
	Let $X\subseteq \mathbb{C}\mathbb{P}^n$ be a non-degenerate, projective variety and $v$ a general point of $X$.
	If $X$ is not a hypersurface then the following holds:
	\begin{enumerate}[label=(\roman*)]
		\item $\deg X_v = \deg X - 1$,
		\item $\dim X = \dim  X_v$,
		\item $X_v$ is non-degenerate,
		\item $\dim I(X_v)_2 \geq \dim I(X)_2-\operatorname{codim X}.$
	\end{enumerate}
	If $X$ is a hypersurface, then $X_v$ is the whole projective space.
\end{thm}

We now discuss some faces of the cones $P_X$ and $\Sigma_X$.
\begin{defi}
	We denote the set of real points of $X$ by $X\left(\mathbb{R}\right)$. Let $v$ be in $X\left(\mathbb{R}\right)$. We set:
	\begin{align*}
		\Sigma_X(v) &:= \left\{ f \in \Sigma_X \, \middle| \, f(v) = 0 \right\}\\
		P_X(v) &:= \left\{f \in P_X \, \middle| \, f(v) = 0 \right\}
	\end{align*}
\end{defi}

It is clear that $\Sigma_X(v)$ and $P_X(v)$ are \textit{exposed faces} of their respective cones, with the point evaluation functional $\ell_v$ definining the supporting hyperplane. The elements of $\Sigma_X(v)$ are sums of squares (mod $I(X)$) of linear forms that vanish on $v$. Crucially the cone
$\Sigma_X(v)$ is actually the cone of sums of squares on the projected variety $X_v$!
To see this, we may assume, after a change of basis, that $v = \left[1:0:\dots : 0\right]$. Then the elements of $\Sigma_X(v)$ are sums of squares of linear forms that do not contain the variable $x_0$. In other words: these elements are sums of squares of linear forms in variables $x_1,\dots,x_n$
modulo $I(X) \cap \mathbb{R}\left[x_1, \dots , x_n\right]$. This elimination on the algebraic side means projection on the geometric side. Therefore we see that $\Sigma_X (v) = \Sigma_{X_v}$.

For nonnegative forms the situations is a bit different: 
\begin{exer} Show that
$P_X(v) \supseteq P_{X_v}$.
\end{exer}

We are now in a position to prove the main theorem of this lecture. Recall from Exercise \ref{ex:degree} that any complex irreducible and non-degenerate variety satisfies the inequality $\deg X \geq \operatorname{codim} X+1$ and the ones where the equality holds are called varieties of minimal degree.

\begin{thm}
Suppose that $X$ is a totally real irreducible non-degenerate variety which is not a variety of minimal degree. Then  $\Sigma_X \subsetneqq P_X$.
\end{thm}
\begin{proof}
If we suppose this is not the case, then we have $\Sigma_X = P_X$ and by the inclusion above we get
\[ \Sigma_{X_v} = \Sigma_X(v) = P_X(v) \supset P_{X_v} \supset \Sigma_{X_v} \]
and therefore $\Sigma_{X_v} = P_{X_v}$.


Projecting away from $\mathrm{codim} X - 1$ many general points in $X(\mathbb{R})$ we get a a hypersurface $X'$
with $\Sigma_{X'} = P_{X'}$. Suppose that $X$ is not of minimal degree. Then 
 $\mathrm{deg} X' \geq 3$, and since $X'$ is a hypersurface, no quadrics vanish on $X'$.
Therefore the forms in $\Sigma_{X'}$ are simply sums of squares of linear forms, since the degree $2$ part of the vanishing ideal of $X'$ is empty.
Thus the forms in $\Sigma_{X'}$ are globally nonnegative.
On the other hand we can find a form in $P_{X'}$ that is not globally nonnegative and we get the desired contradiction. The details are explained in the following exercise.
\end{proof}

\begin{exer}
Let $X$ be a totally real hypersurface in $\mathbb{P}^n$ of degree at least $3$. Show that there exists a quadratic form $Q$ that is nonnegative on $X$, but $Q$ is not globally nonnegative. (Hint: Project $X\left(\mathbb{R}\right)$ onto the unit sphere $S$. Take $v$ outside of the image $Y$ of $X\left(\mathbb{R}\right)$.
Since $X\left(\mathbb{R}\right)$ is closed, $Y$ is compact. Make a quadratic form that is negative on a small neighborhood of $v$ and positive everywhere else).
\end{exer}
We can use the idea of sequentially projecting away from general points to prove some classical results in algebraic geometry.

\begin{thm}[Castelnuovo's bound] 
	Let $X$ be an irreducible non-degenerate projective variety. Then:
	\[ \dim I(X)_2 \leq \binom{\mathrm{codim} X + 1}{2}. \]
	Furthermore, equality holds if and only if $X$ is a variety of minimal degree.
\end{thm}
\begin{proof}[Proof sketch.]
	When $X$ is a hypersurface, the inequality is clear. And equality holds if and only if $X$ is a quadratic hypersurface, which is a variety of minimal degree. Now we reduce to the hypersurface case by repeatedly projecting away from general points.
	Let $v$ be a general point of $X$. By Theorem \ref{thm:proj} part (4) the ideal of $X_v$ loses at most $\operatorname{codim} X$ quadrics compared to the ideal of $X$.
	\[ \dim I(X)_2  - \dim I(X_v)_2  \leq \operatorname{codim} X.\]
	In total we take $\operatorname{codim} X-1$ projections to get to the hypersurface case where we lose at most $1$ additional quadric, so the total number of lost quadrics is at most $1+2+\dots+\operatorname{codim} X$ as desired. Furthermore, equality only holds if projecting away from $\operatorname{codim} X-1$ general points of $X$ we get a quadratic hypersurface. Then by Theorem \ref{thm:proj} parts (1)-(3) $X$ is a variety of minimal degree.
	We do not show that equality indeed holds for all varieties of minimal degree.
\end{proof}
The paper \cite{Pyth} introduced quadratic persistence of a variety to measure how quickly quadrics are lost when successively projecting away from points of a variety $X$. This quantity also gives us a bound on the Pythagoras number of a real projective variety (See Definition \ref{def:Pyth} and Theorem \ref{thm:qp-pyth} below).
\begin{defi}
	The \textit{quadratic persistence} of a variety $X$, denoted $\operatorname{qp}\left(X\right)$,
	is defined as the least number $k$ of points $v_1, \dots , v_k \in X$ such that
	$X_{v_1\dots v_k}$ has no quadratics in its vanishing ideal.
\end{defi}
\begin{rmk} Observe that $\Sigma_{X_{v_1 \dots  v_k}} = \Sigma_X\left(v_1, \dots , v_k\right)$ \end{rmk}

The definition of quadratic persistence allows an arbitrary subset of points from which to project. It is an important observation that for irreducible varieties $X$ projection away from a general subset of points will lose the maximal number of quadrics, and therefore it suffices to consider general subsets to figure out quadratic persistence.

\begin{lem}\cite{Pyth}
	Choosing general points $v_1, \dots , v_k$ on an irreducible variety $X$ computes the quadratic persistence. 
\end{lem}

Quadratic persistence is not well-understood for many varieties. For instance it is an open question to compute the quadratic persistence of Veronese embeddings of $\mathbb{P}^n$ for $n>2$. This question also has a relation to the dimension of certain catalecticant varieties \cite{MR1735271}.
For $n=2$ answer is known by work of \cite{MR1735271} and also proved in \cite{Pyth}. 
\begin{thm}
$\operatorname{qp}(\nu_d(\mathbb{P}^2))=\binom{d+1}{2}.$
\end{thm}

For $n>2$ we have the following conjecture (cf. \cite{MR3467328} and \cite{MR3417682}), which, if true, allows one to quickly compute quadratic persistence of Veronese embeddings of $\mathbb{P}^n$: 
\begin{conj}
	In each projection step we lose the maximal number of quadrics, which is given by the codimension. This conjecture holds for $n=2$.
\end{conj}

We illustrate this conjecture by an example (the conjecture for this example can be experimentally verified by a suitable choice of points):

\begin{ex}
Consider $X=\nu_3(\mathbb{P}^3)$. Then $\dim R(X)_1=\binom{6}{3}=20$, and so $X\subset \mathbb{P}^{19}$ with $\operatorname{codim} X=16$. Moreover,  $\dim R(X)_2=\binom{9}{6}=84$, while $\dim I(X)_2=\binom{21}{2}-\dim R(X)_2=126$. Losing the maximal number of quadrics at each step means that we lose $16$ quadrics in the first projection, $15$ quadrics in the second projection, $14$ quadrics in the third and so on. We have $16+15+14+\dots+7+6+5=126$, and so the conjecture says that $\operatorname{qp} \nu_3(\mathbb{P}^3)=12$.
\end{ex}
We now explain the relevance of quadratic persistence to sums of squares.

\begin{defi}\label{def:Pyth}
	The \textit{Pythagoras number} of a variety $X$, denoted $\Pi\left(X\right)$, is the least number $k$ such that any sum of squares in $\Sigma_X$
	can be written as a sum of $k$ squares.
\end{defi}

\begin{rmk}
	In the graph/matrix case the Pythagoras number is equal to the least number $k$ such that any partially specified matrix, which is completeable to a positive semidefinite matrix, can be completed to a rank $k$ positive semidefinite matrix. Pythagoras number in this case has been called the \textit{Gram dimension} of the graph $G$ \cite{MR3006041}.
\end{rmk}

\begin{thm}\label{thm:qp-pyth}
	Let $X$ be a non-degenerate, totally real, irreducible variety in $\mathbb{P}^n$. Then:
	\[ \Pi\left(X\right) \geq n + 1 - \operatorname{qp}\left(X\right). \]
\end{thm}
\begin{proof}
	Let $v_1, \dots v_k$ be general points where $k = \operatorname{qp}\left(X\right)$.
	Then $\Sigma_X\left(v_1 , \dots , v_k\right) \cong \mathcal{S}_+^{n+1-k}$ holds.
	On this face we need $n+1-k = n + 1 - \operatorname{qp}\left(X\right)$ many squares.
\end{proof}

We currently do not have any examples where the above bound is not tight.

\begin{problem}[Open Question]
	Find a variety $X$ such that $\Pi\left(X\right) > n+1 - \operatorname{qp}\left(X\right)$.
\end{problem}
 Although the quadratic persistence of Veronese embeddings of $\mathbb{P}^2$ is known, the Pythagoras number is not completely understood. It is known that $\Pi(\nu_3(\mathbb{P}^2))=4$. For higher degrees the following result of Scheiderer almost settles the question. The lower bound on the Pythagoras number comes from quadratic persistence.
\begin{thm} \cite{MR3648509}
	Let $d \geq 4$. Then:
	\[ \Pi\left(\nu_d\left(\mathbb{P}^2\right)\right) = \begin{cases} d+1 \text{, or} \\ d+2. \end{cases} \]
\end{thm}

\section{Lecture 3}

The main motivating question for this lecture is the following:
\begin{center}What is the relationship between $P_X$ and $\Sigma_X$ when they are not equal?\end{center}
We will address this question in two main ways: a direct quantitative comparison and comparison via the dual cones.

\subsection{Globally Nonnegative Polynomials and Sums of Squares}
One possible way of comparing the relative size of two cones is to slice both cones with an affine hyperplane in such a way that both slices are compact convex sets and then compare the sections, for instance via volume. If we want to think of convex sets $K$ and $(1+\varepsilon)K$ as having roughly equal size for small $\varepsilon$ then volume is not a good measure of size, when the dimension $D$ of $K$ is large, since $\operatorname{vol} (1+\varepsilon)K=(1+\varepsilon)^D \operatorname{vol} K$. A good measure to use is $(\operatorname{vol} K)^{1/D}$.

For the cones $\Sigma_{n,2d}$ and $P_{n,2d}$ a good choice of slicing hyperplane is given by the affine hyperplane of polynomials of average $1$ on the unit sphere $\mathcal{S}^{n-1}$ in $\mathbb{R}^n$. Let $\bar{\Sigma}_{n,2d}$ and $\bar{P}_{n,2d}$ be the corresponding slices. We can compare the relative size of these slices in the asymptotic regime where the degree $d$ is fixed and the number of variables $n$ goes to infinity (See \cite{MR2254649} and \cite[Chapter 4]{BPT}).
\begin{thm}\label{thm:vol}
Let the degree $2d$ be fixed. There exist constants $c_1$ and $c_2$ independent of $n$ such that:
$$c_1n^{\frac{d-1}{2}}\leq\left( \frac{\operatorname{vol} \bar{P}_{n,2d}}{\operatorname{vol} \bar{\Sigma}_{n,2d}}\right)^{1/D}\leq c_2n^{\frac{d-1}{2}}.$$
\end{thm}
We see therefore that if the degree $2d$ is at least  $4$ and the number of variables $n$ is sufficiently large, the cone $P_{n,2d}$ of nonnegative polynomials is significantly larger then the cone of sums of squares $\Sigma_{n,2d}$.

We now start analyzing the relation between nonnegative polynomials and sums of squares using the dual cones $P_X^*$ and $\Sigma_X^*$.
\subsection{Dual Cones}
We begin with some preliminaries on duality in convex geometry. Let $V$ be a finite dimensional vector space and $V^*$ be the dual vector space of linear functionals on $V$.
Further let $C$ be a convex cone and $$C^* = \left\{ \ell \in V^* \, \middle| \, \ell\left(v\right) \geq 0, \hspace{.2 cm}\text{for all} \hspace{.2 cm} v \in C \right\}$$ be its dual cone.


\begin{rmk}
	The following theorem is known as biduality (or bipolarity) theorem \cite[Chapter IV]{MR1940576}: the double-dual cone $\left(C^*\right)^*$ is the closure of $C$. If $C$ is closed then $\left(C^*\right)^* = C$.
\end{rmk}

Let $X$ be a totally real variety. The dual space $R_2^*$ contains a distinguished set of linear functionals. For any point $v \in X$ we can pick an affine representative $\tilde{v} \in \mathbb{R}^{n+1}$, and evaluation of forms in $R_2$ on $\tilde{v}$ produces a linear functional from $R_2$ to $\mathbb{R}$. Since the choice fo affine representative only changes the resulting linear functional by a constant, we will mildly abuse notation and call the resulting linear functional $\ell_v$, without specifying an affine representative $\tilde{v}$.

We set
\[ A_X = \operatorname{Conical Hull}\left\{ \ell_v \, \middle| \, \ell_v \text{ is a point evaluation on } v \text{ for } [v] \in X \right\} .\]

\begin{exer}
Show that $A_X$ is a closed convex cone in $R_2^*$.
\end{exer}
It is now not hard to show that $A_X$ is in fact equal to $P_X^*$.
\begin{lem} We have $A_X={P}_X^*$.
\end{lem}
\begin{proof}
	The fact that $A_X^* = {P}_X$ follows tautologically from the definition of what it means to be nonnegative on $X$. Using biduality we see that
	$A_X=\operatorname{closure}\left(A_X\right) = \left(A_X^*\right)^* = {P}_X^*$.
\end{proof}

\begin{rmk}
	A rich connection to classical real analysis comes from the fact that the dual to nonnegative polynomials on a subset $K$ of $\mathbb{R}^n$ are linear functionals that can be written as integration with respect to a measure supported on $K$. Understanding what linear functionals come from a measure is known as a \textit{moment problem}. If we take a finite dimensional vector space of polynomials on $K$, then this is known as a \textit{truncated moment problem} \cite{MR3729411}.
	
\end{rmk}

\begin{defi}
	A section of the cone $\mathcal{S}^{n+1}_+$ with an (affine) subspace is called a \textit{spectrahedron}. If the subspace is linear then the section is a \textit{spectrahedral cone}. 
\end{defi}

\begin{exer}\label{exer:eray}
Show that a matrix $M$ spans an extreme ray of the cone $\mathcal{S}^{n+1}_+$ of positive semidefinite matrices if and only if $M$ is rank $1$ and positive semidefinite. Give an example of a section of $\mathcal{S}^3_+$ with a linear subspace such that the resulting spectrahedral cone has extreme rays that are not of rank $1$.
\end{exer}

We can think of a linear functional $\ell \in R_2^*$ as a quadratic form on $R_1$, or equivalently a $(n+1)\times(n+1)$ symmetric matrix:

\begin{defi}
	To any $\ell \in R_2^*$ we can associate a quadratic form
	\[ Q_\ell : R_1 \rightarrow \mathbb{R} \quad Q_\ell\left(f\right) = \ell\left(f^2\right). \]
	More formally the above defines an injective linear map $\varphi_X:R_2^* \rightarrow \operatorname{Sym}^2 R_1^*,$ where $ \operatorname{Sym}^2 R_1^*$ is the vector space of quadratic forms on $R_1$, via $\varphi_X(\ell)=Q_\ell$. By picking a basis of $R_1$ we can think of $ \operatorname{Sym}^2 R_1^*$ as the vector space of $(n+1)\times (n+1)$ symmetric matrices.
\end{defi}

The following then holds by definition of $\Sigma_X^*$: \[ \Sigma_X^* = \left\{ \ell \in R_2 \, \middle| \, Q_\ell \text{ is positive semidefinite} \right\}. \]
From this we see that $\Sigma_X^*$ is a spectrahedral cone. Namely, $\operatorname{Sym}^2 R_1^*$ contains the cone $\mathcal{S}^{n+1}_+$ of positive semidefinite quadratic forms, and $\Sigma_X^*$ is the intersection of $\mathcal{S}^{n+1}_+$ with $\varphi_X(R_2^*)$. We call the dual cone $\Sigma_X^*$ the \textit{Hankel spectrahedron} of $X$.

For a point evaluation $\ell_v$ we get
\[ Q_{\ell_v}\left(f\right) = \ell_v\left(f^2\right) = f^2\left(v\right) = \left(f\left(v\right)\right)^2. \]
Therefore we see that the resulting quadratic form has rank $1$, as it is the square of a linear functional.

\begin{exer}
Suppose that for some $\ell \in R_2^*$ quadratic form $Q_{\ell}$ has rank $1$. Show that then $\ell$ is a point evaluation functional: $\ell=\pm \ell_v$ for some $v \in X$.
\end{exer}

The above discussion gives a very nice geometric interpretation of the dual cones $\Sigma_X^*$ and $P_X^*$: as we observed the cone $\Sigma_X^*$ is the intersection of the cone of positive semidefinite matrices $\mathcal{S}^{n+1}_+$ with the linear subspace $\varphi_X(R_2^*)$. The dual cone $P^*_X$ is convex hull of positive semidefinite rank $1$ matrices that are in $\varphi_X(R_2^*)$. 

\begin{ex}
Let $X=\nu_d(\mathbb{P}^1)$. Then we can identify $R(X)_1$ with degree $d$ bivariate forms and $R(X)_2$ with degree $2d$ bivariate forms. We consider the vector spaces $R_1$ and $R_2$ with respect to the usual monomial bases consisting of monomials $t^d$, $st^{d-1}, \dots, s^d$ and  $t^{2d}$, $st^{2d-1}, \dots, s^{2d}$. Consider the map $\varphi_X$ associating to a linear functional $\ell \in R_2^*$ the $(d+1)\times (d+1)$ symmetric matrix $M_\ell$ of the quadratic form $Q_\ell$. The $(i,j)$-entry of $M_\ell$ is equal to $\ell(s^it^{d-i} \cdot s^jt^{d-j})=\ell(s^{i+j}t^{2d-i-j})$. We see that the $(i,j)$-entry of $Q_\ell$ depends only on the sum $i+j$, and therefore $M_\ell$ is constant along antidiagonals. Such a matrix is called a \textit{Hankel matrix}. It is not hard to show that $\varphi(R_2^*)$ is the full subspace of $(d+1)\times (d+1)$ Hankel matrices. 

Therefore the cone $\Sigma_{\nu_d(\mathbb{P}^1)}^*$ is the cone of $(d+1)\times (d+1)$ positive semidefinite Hankel matrices. Since $\Sigma_{\nu_d(\mathbb{P}^1)}=P_{\nu_d(\mathbb{P}^1)}$ we see that the extreme rays of $\Sigma_{\nu_d(\mathbb{P}^2)}^*$ are precisely rank $1$ positive semidefinite Hankel matrices. This fact can be used to solve the Hamburger moment problem in real analysis (and its truncated version)  \cite{MR3729411}.
\end{ex}

The above reasoning leads immediately to the following Proposition:

\begin{prop}
We have equality $P_X=\Sigma_X$ if and only if the dual cone $\Sigma_X^*$ only has rank $1$ extreme rays.
\end{prop}
Extreme rays of the cone $\mathcal{S}^{n+1}_+$ are just positive semidefinite quadratic forms of rank 1 (see Exercise \ref{exer:eray}). Therefore we can interpret the above proposition as follows: we have equality between $\Sigma_X$ and $P_X$ if and only if all extreme rays of the slice $\Sigma_X^*=\mathcal{S}^{n+1}_+\cap \varphi_X(R_2^*)$ are extreme rays of the original cone $\mathcal{S}^{n+1}_+$.

When the cones $P_X$ and $\Sigma_X$ are not equal we can examine the ranks of extreme rays of $\Sigma_X^*$ and this can give us some picture of the difference between nonnegative polynomials and sums of squares. A complex projective variety is called a variety of \textit{almost minimal degree} if it satisfies $\deg X=\operatorname{codim} X+2$. A variety is called \textit{arithmetically Cohen-Macaulay} if its homogeneous coordinate ring is Cohen-Macaulay (see \cite{MR1322960} for more on Cohen-Macaulay rings). The following was shown in \cite{MR3486176}:

\begin{thm}\label{thm:hiamd}
Let $X\subset \mathbb{P}^n$ be a totally real, arithmetically Cohen-Macaulay, irreducible variety of almost minimal degree. Then $\Sigma_X^*$ has extreme rays of rank $1$ and $\operatorname{codim} X$ only.
\end{thm}
In particular this theorem applies to $X=\nu_3(\mathbb{P}^2)$ and $X=\nu_2(\mathbb{P}^3)$ and thus generalizes the result of \cite{MR2904568} where this interesting behavior of extreme rays of $\Sigma_X^*$ was observed for $\Sigma^*_{3,6}$ and $\Sigma^*_{4,4}$. Some interesting consequences for the \textit{algebraic boundary} of the cones $\Sigma_{3,6}$ and $\Sigma_{4,4}$ were found in \cite{MR2999301}.

In \cite{MR3550352} a detailed study of extreme rays of Hankel spectrehedra for ternary forms was undertaken. Even then, the possible ranks of extreme rays of $\Sigma_{3,2d}^*$ are not known for $d\geq 8$. As we will see in the next section the question of ranks of extreme rays of $\Sigma_X^*$ is also hard for monomial ideals.
\begin{problem}[Open Question]
	Let $X$ be a totally real projective variety. What are the ranks of extreme rays of the Hankel spectrahedron of $X$? What can be said for the case where $X$ is the Veronese embedding of $\mathbb{P}^n$?
\end{problem}

We call the lowest rank of extreme ray of $\Sigma_X^*$ that is strictly greater than $1$ the \textit{Hankel index} of $X$. In \cite{MR3633773} a very intriguing connection between the Hankel index and the free resolution of the ideal $I(X)$ was found (recall discussion of free resolutions in Lecture \ref{sec1}):

\begin{thm}\label{thm:hindex}
Let $X$ be a totally real projective variety with Green-Lazarsfeld index $p$. Then the Hankel index of $X$ is at least $p+1$.
\end{thm}
\begin{rmk}
Theorem \ref{thm:hiamd} can be deduced from Theorem \ref{thm:hindex}, since it is known that arithmetically Cohen-Macaulay varieties of almost minimal degree have Green-Lazarsfeld index $\operatorname{codim} X-1$ \cite{MR2946932}. Therefore we know that the Hankel index is at least $\operatorname{codim} X$. As we will explain in the next section (see Proposition \ref{prop:ker}) we know that for any variety that is not of minimal degree the Hankel index is at most $\operatorname{codim} X$ and the result follows.
\end{rmk}

\subsubsection{Under the hood.} We briefly discuss some details of the proofs of the results on extreme rays of $\Sigma_X^*$. The following general Lemma whose proof is left as an exercise is quite useful for this investigation.
\begin{lem}\label{lem:exray}\cite{MR1342934}
Let $\mathcal{S}^n_+$ be the cone of $n \times n$ positive semidefinite matrices and $C=\mathcal{S}^n_+\cap L$ be a spectrahedral cone. Then $M\in C$ spans an extreme ray of $C$ if and only if the kernel of $M$ is maximal over all matrices in $L$: if $\ker M \subseteq \ker A$ for some $A \in L$ then $A=\lambda M$ for some $\lambda \in \mathbb{R}$. 
\end{lem}
We say that a subspace $W$ of $R_1$ is a \textit{basepoint-free linear series} on $X$ if the linear subspace of $\mathbb{C}\mathbb{P}^n$ defined by the vanishing of all linear forms in $W$ does not intersect the complex variety $X_{\mathbb{C}}$. The characterization of Lemma \ref{lem:exray} can be used to derive the following Proposition:
\begin{prop}\label{prop:ker}\cite{MR3633773}
Suppose a linear functional $\ell$ spans an extreme ray of $\Sigma_X^*$ and $\operatorname{rank} Q_\ell >1$. Let $W$ be the kernel of $Q_\ell$.
Then the following hold:
\begin{enumerate}
\item $W$ is a basepoint-free linear series on $X$.
\item Consider the ideal $\langle W\rangle$ generated by $W$ in $R$. Then $\langle W\rangle_2$ is a codimension $1$ subspace of $R_2$.
\end{enumerate}


\end{prop}

We observe that from part (1) of the above Proposition it follows that the dimension of the kernel of any extreme ray of $\Sigma_X^*$ is at least $\dim X+1$. In particular, the Hankel index of $X$ is at most $\operatorname{codim} X$.

 We now use a result of \cite{MR2230917} to connect the Hankel index to free resolutions.
 
\begin{thm}\label{thm:n2p}
Suppose that the Green-Lazarsfeld index of $X$ is $p$. Then for any basepoint-free linear series $W\subseteq R_1$ such that $\dim W >\dim R_1-p$ we have $\langle W\rangle_2=R_2$.
\end{thm}

Putting together Proposition \ref{prop:ker} and Theorem \ref{thm:n2p} we obtain the following:

\begin{cor}\cite{MR3633773} Let $X$ be a totally real projective variety. Then the Hankel index of $X$ is at least the Green-Lazarsfeld index of $X$ plus one. 
\end{cor}

Some cases where the above bound is sharp are discussed in \cite{MR3633773} and we will see an application of one such case in the next subsection. Currently the cases where the Hankel index is equal to the Green-Lazarsfeld index plus one are not well-understood. 

\begin{ex}
We revisit the smallest cases in Hilbert's theorem where nonnegative polynomials are not sums of squares. Here we have $X=\nu_3(\mathbb{P}^2)$ or $X=\nu_2(\mathbb{P}^4)$. 
By Theorem \ref{thm:hiamd} we know that in both cases $\Sigma_X^*$ has extreme rays of rank $1$ and $\operatorname{codim} X$ only. Specifically, when $X=\nu_3(\mathbb{P}^2)$ the cone $\Sigma_X^*$ has extreme rays of rank $1$ and $7$ only, and $X=\nu_2(\mathbb{P}^4)$ has extreme rays of rank $1$ and $6$ only. We know that extreme rays of rank $7$ and $6$ must exist in the cases since $\Sigma_X\subsetneq P_X$. However, this still leaves the question of how to construct such extreme rays. This is what we will address next. For more on fascinating geometry of the minimal cases in Hilbert's theorem see \cite{MR2999301}.
\end{ex}
\subsubsection{Constructing extreme rays of the dual cone $\Sigma_X^*$.} We can also see a relation between equality of $\Sigma_X$ and $P_X$ and varieties of minimal degree in a different way from Lecture 2. Suppose that $X$ is not a variety of minimal degree and consider the intersection $Y$ of $X$ with a general linear space $L$ such that $\dim L =\operatorname{codim} X$. By Bertini's theorem \cite[Theorem 17.22]{MR1416564} we know that $Y$ is the union of $\deg X$ many points. Moreover, since $\deg X>\operatorname{codim} X+1$ and $Y$ lies in $L$ we see that the points of $Y$ are not linearly independent.

A dual way of saying this is that evaluations of linear forms on the points of $Y$ are not independent. In fact there are \textit{linear relations} among values of linear forms on the points of $Y$. One can leverage such linear relations into building a linear functional $L\in R_2^*$ which separates $\Sigma_X$ from $P_X$. We will not go into all of the technical details, but will instead give an illustrative example for the case of ternary sextics, i.e. when $X=\nu_3(\mathbb{P}^2)$. A hyperplane section of $X$ corresponds to a planar cubic curve, and so a section of $X$ with a subspace $L$ of  codimension $2$ corresponds to the intersection of two planar cubics in $9$ points. We will pick the cubics so that the nine intersection points are real. The idea of using such a complete intersection to show that $\Sigma_{3,6}$ is not equal to $P_{3,6}$ goes all the way back for to Hilbert's original proof. Later Robinson used this idea to build an explicit example of a nonnegative form that is not a sum of squares \cite{MR1747589}.
\subsubsection{Ternary Sextics.}
Let $q_1$ and $q_2$ be two cubics intersecting in $9$ real points. One can take each cubic to be the union of $3$ real lines.
Consider the evaluation map:
\begin{align*}
	\mathrm{Ev} : \overbrace{\mathbb{R}\left[x\right]_3}^{\cong \mathbb{R}^{10}} &\rightarrow \mathbb{R}^9\\
	f &\mapsto \left(f\left(v_1\right), \dots , f\left(v_9\right)\right)
\end{align*}

The evaluation map is a linear map. The kernel of $\operatorname{Ev}$ consists of cubic forms vanishing on $v_1,\dots,v_9$. It is possible to show that $q_1$ and $q_2$ generate a radical ideal and therefore the dimension of the kernel of $\operatorname{Ev}$ is $2$ \cite{MR1376653}. By elementary linear algebra, the dimension of the image of $\operatorname{Ev}$  is $8$, and therefore the image of $\operatorname{Ev}$ is a hyperplane in $\mathbb{R}^9$. Therefore, there exists a linear functional on $\mathbb{R}^9$ which vanishes on the image $\operatorname{Im} \operatorname{Ev}$.

More specifically this means that there exist $\lambda_1, \dots , \lambda_9 \in \mathbb{R}$ such that $\lambda_1 f\left(v_1\right) + \dots + \lambda_9 f\left(v_9\right) = 0$ for all $f \in \mathbb{R}\left[x\right]_3$. In other words the values of cubics at these nine points are \textit{linearly dependent}. This would not happen if the points $v_1,\dots,v_9$ were generic, but they are not, they come from a complete intersection of two cubics. The linear relation on the values of cubics on $v_1,\dots,v_9$ is called \textit{Cayley-Bacharach relation}. Additionally it is possible to show that the coefficients $\lambda_1,\dots,\lambda_9$ are non-zero, i.e. the relation must involve all the $9$ points in the intersection of the cubics. For more on the history of Cayley-Bacharach theorem, relations and generalizations see \cite{MR1376653}.

We can build a linear functional $L$ in $\Sigma_X^*$ but not in $P_X^*$ by considering linear combinations of the point evaluation functionals $\ell_{v_1}, \dots , \ell_{v_9}$. We write $L := \mu_1 \ell_{v_1} + \dots + \mu_9 \ell_{v_9}$ so that the associated quadratic form is  $Q_L\left(f\right) = \mu_1 f\left(v_1\right)^2 + \dots + \mu_9 f\left(v_9\right)^2$. 

Recall that point evaluations correspond to rank $1$ quadratic forms. In this case we are adding squares of $9$ linear functionals (point evaluations). Our previous discussion shows that the span of point evaluations on ternary cubics is eight dimensional. The fact that in the unique Cayley-Bacharach relation all the coefficients are non-zero indicates that any eight of the point evaluations on cubics are linearly independent. This implies that eight of the nine coefficients $\mu_i$ must be positive and at most one of $\mu_i$ is negative, since $Q_L\left(f\right) = \mu_1 f\left(v_1\right)^2 + \dots + \mu_9 f\left(v_9\right)^2$ is nonnegative on squares of cubics. If all coefficients $\mu_i$ are positive, then $L$ lies in $P^*_X$. Therefore we want eight positive coefficients $\mu_i$ and one negative $\mu_i$. Without loss of generality we may assume that $\mu_1, \dots , \mu_8 >0 , \mu_9 < 0$.

As shown in \cite{MR2904568} given positive $\mu_1,\dots,\mu_8$ we may choose any $\mu_9<0$ such that 
$$\mu_9\geq -\frac{\lambda_9}{\sum_{i=1}^8 \frac{\lambda_i^2}{\mu_i}},$$ and the quadratic form $Q_\ell$ will be positive semidefinite. Furthermore the minimal choice of $\mu_9$, which makes the above inequality into equality will lead to a linear functional $\ell \in \Sigma_{3,6}^*$ and $\ell \notin P_{3,6}^*$. In fact, it can be shown that $\ell$ will be an extreme ray of $\Sigma_{3,6}^*$.


\subsection{Quadratic Monomial Ideals}

When the graph $G$ is not chordal one can investigate what additional conditions besides the ``obvious" condition that every filled-in principal submatrix is positive semidefinite are necessary and sufficient for finding a positive semidefinite completion. 

\begin{ex}
	 Geometrically the variety $X(C_n)$ of the $n$-cycle $C_n$ is the union of $n$ lines in $\mathbb{P}^{n-1}$.
	 On the matrix completion side this corresponds to matrices of the following form:
	 \[ \begin{pmatrix}
	 		*	& *	& 	& ?	& 	& * 	\\
	 		*	& *	& *	& 	&	& 	\\
	 			& *	& *	& 	& ? 	&	\\
	 		?	&	&	& \ddots	& 	&	\\
	 			&	& ?	&	&	&	\\
			*	&	&	&	& *	& *
	 	\end{pmatrix}
	 \]
\end{ex}
For the $n$-cycle $C_n$, ($n\geq 4$) the ``obvious" necessary conditions for PSD completability are not sufficient, however the necessary and sufficient conditions are not hard to find \cite{MR1428642}. It is natural to ask for what graphs does positive semidefiniteness of all fully specified submatrices and completability of all induced cycles suffice. This was answered in \cite{MR1342017}: the graphs for which cycle completability and the ``obvious" necessary condition are sufficient are clique sums of chordal and \textit{series parallel} graphs.  

The question of possible ranks of extreme rays of the cone $\Sigma_G^*$ has also been analyzed extensively. The largest rank of an extreme ray of $\Sigma_G^*$ was called \textit{the sparsity order} of $G$ in \cite{MR960140}.
A result of Laurent \cite{MR1863577} describes all graphs $G$ for which $\Sigma_G^*$ only has rank $1$ and $2$ extreme rays solving a conjecture of  \cite{MR960140}. 

Using Theorem \ref{thm:hindex} and a result of \cite{MR2188445} on the Green-Lazarsfeld index of quadratic square-free monomial ideals we obtain the following Theorem which finds the smallest rank of an extremal ray of $\Sigma_G$ that is greater than $1$:

\begin{thm}\cite{MR3633773}
Let $G$ be a non-chordal graph. Then the lowest rank of an extreme ray of $\Sigma_G^*$ is equal to the length of the smallest induced cycle in $G$ minus $2$.
\end{thm}

\subsection{Quantitative Comparison of $P_G$ and $\Sigma_G$.}
	Let $\bar{P}_G$ and $\bar{\Sigma}_G$ be the intersections of the respective cones $P_G$ and $\Sigma_G$ with the trace $1$ hyperplane. As a quantitative comparison of sizes of $P_G$ and $\Sigma_G$ we can consider the \textit{expansion ratio} $\alpha(G)$ of $\Sigma_G$ and $P_G$ with respect to the scaled identity matrix $\frac{1}{n}I_n$: let $\alpha(G)$ be the smallest real number such that $\bar{P}_G \subseteq I_n+\alpha (\bar{\Sigma}_G-\frac{1}{n}I_n)$. Geometrically $\alpha(G)$ measures how much we need to expand $\bar{\Sigma}_G$ around $\frac{1}{n}I_n$ so that the expansion contains $P_G$.

\begin{thm}\cite{BS2020}\label{thm:bshu}
	For $n$-cycles we have $\alpha(G) =1 + O(\frac{1}{n^2})$ as $n \rightarrow \infty$.
\end{thm}

\begin{rmk}
In analogy with Theorem \ref{thm:vol} we can also consider the volume ratio $$\left(\frac{\operatorname{vol} \bar{\Sigma}_G}{\operatorname{vol}\bar{P}_G}\right)^{\frac{1}{D}},$$ with $D=\binom{n+1}{2}-1$. The above result implies that the volume ratio $$\left(\frac{\operatorname{vol} \bar{\Sigma}_{C_n}}{\operatorname{vol}\bar{P}_{C_n}}\right)^{\frac{1}{D}}$$ goes to $1$ as $n$ approaches infinity, confirming an experimental observation of \cite{MR2740626}.
\end{rmk}

\begin{rmk} Using clique sums the result of Theorem \ref{thm:bshu} has been extended to a more general family of graphs than simply cycles. However a small induced cycle will always cause the expansion ratio to be large.

\end{rmk}

It is natural to ask whether a quantitative comparison of sizes of cones $\Sigma_X$ and $P_X$ can be developed for general varieties $X$. A good candidate would be arithmetically Cohen-Macaulay varieties of almost minimal degree. In fact, cycle graphs can be viewed as degenerations of elliptical normal curves, and they are actually the natural graph analogues of such varieties. However, there are currently no results in this direction.

\begin{problem}
Develop methods for quantitative comparisons of the cones $\Sigma_X$ and $P_X$ for general totally real varieties $X$.
\end{problem}
It is tempting to try to leverage graph results to general varieties by using \textit{degeneration techniques}, such as Gr\"{o}bner degeneration. For instance, the graph case of Theorem \ref{thm_reg_2} is Theorem \ref{thm:graphs}, which is significantly easier to prove. We end the lecture notes with the following open-ended question:

\begin{problem}[Open Question]
	Can we use (Gr\"obner) degenerations to prove results on relationship between nonnegative polynomials and sums of squares results?
\end{problem}

\textbf{Ackhowledgements.} We would like to thank Daniele Agostini, Thomas Kr\"{a}mer, Marta Panizzut and Rainer Sinn  for organizing a wonderful Fall School. We are grateful to Rainer Sinn for help in improving these notes. Grigoriy Blekherman was partially supported by NSF grant DMS-1901950.

	\bibliographystyle{alpha}
	\bibliography{ChapterBib}
\end{document}